\documentclass[11pt]{article}

\usepackage{amsthm}
\usepackage{latexsym}
\usepackage{amssymb}

\newtheorem{thm}{Theorem}
\theoremstyle{remark}

\newtheorem{exm}{Example}[section]

\begin{document}
{

\begin{center}
{\Large\bf
The strong matrix Stieltjes moment problem.}
\end{center}

\begin{center}
{\bf A.E. Choque Rivero, S.M. Zagorodnyuk}
\end{center}

\section{Introduction.}
In this paper we analyze the following problem:
to find a non-decreasing matrix function
$M(x) = ( m_{k,l}(x) )_{k,l=0}^{N-1}$, on $\mathbb{R}_+ = [0,+\infty)$, $M(0)=0$,
which is left-continuous on $(0,+\infty)$, and such that
\begin{equation}
\label{f1_1}
\int_{ \mathbb{R}_+ } x^n dM(x) = S_n,\qquad n\in \mathbb{Z},
\end{equation}
where $\{ S_n \}_{n\in \mathbb{Z}}$ is a prescribed sequence of Hermitian $(N\times N)$ complex matrices (moments),
$N\in \mathbb{N}$. This problem is said to be a {\bf strong matrix Stieltjes moment problem}.
The problem is said to be {\bf determinate} if it has a unique solution and
{\bf indeterminate} in the opposite case.

\noindent
The scalar ($N=1$) strong Stieltjes moment problem (in a slightly different setting) was
introduced in 1980 by Jones, Thron and Waadeland~\cite{cit_10_JTW}.
Necessary and sufficient conditions for the existence of a solution with an infinite number of
points of increase and for the uniqueness of such a solution were established~\cite[Theorem 6.3]{cit_10_JTW}.
Also necessary and sufficient conditions for the existence of a unique solution with a finite
number of points of increase were obtained~\cite[Theorem 5.2]{cit_10_JTW}.
The approach of Jones, Thron and Waadeland's investigation was made through the study of
special continued fractions related to the moments.

\noindent
In 1995, Nj\aa stad described some classes of solutions of the scalar strong Stieltjes moment
problem~\cite{cit_15_N},\cite{cit_30_N}. He used properties of some associated Laurent polynomials.

\noindent
In 1996, Kats and Nudelman obtained necessary and sufficient conditions for the existence of a solution
of the scalar strong Stieltjes moment problem~\cite[Theorem 1.1]{cit_35_KN}.
The degenerate case was studied in full:
in this case the solution is unique, given explicitly and it has a finite number of points of increase.
In the non-degenerate case, conditions for the determinacy were given and the unique solution was
presented. In the (non-degenerate) indeterminate case a Nevanlinna-type parameterization for all solutions
of the scalar strong Stieltjes moment problem was obtained~\cite[Theorem 4.1]{cit_35_KN}.
Canonical solutions and Weyl-type lunes were studied, as well.
Kats and Nudelman used the results of Krein on the semi-infinite string theory.

\noindent
Various other results on the scalar strong Stieltjes moment problem can be found in
papers~\cite{cit_40_N},\cite{cit_50_JTW},\cite{cit_60_JT},\cite{cit_65_N},\cite{cit_67_N}
(see also References therein).

\noindent
The moment problem~(\ref{f1_1}) where the half-axis $\mathbb{R}_+$ is replaced by the whole
axis $\mathbb{R}$ is said to be the {\bf strong matrix Hamburger moment problem}.
The scalar ($N=1$) strong matrix Hamburger moment problem has been intensively studied since 1980-th,
see a survey~\cite{cit_60_JT}, a recent paper~\cite{cit_70_BD} and References therein.
For the matrix case, see papers~\cite{cit_80_S},\cite{cit_90_Z} and papers cited there.

The aim of our present investigation is threefold. Firstly, we obtain necessary and sufficient
conditions for the solvability of the strong matrix Stieltjes moment problem~(\ref{f1_1}).
Consider the following block matrices constructed by moments:
\begin{equation}
\label{f1_2}
\Gamma_n = (S_{i+j})_{i,j=-n}^n =
\left(
\begin{array}{ccccc} S_{-2n} & \ldots & S_{-n} & \ldots & S_0\\
\vdots &  & \vdots &  & \vdots\\
S_{-n} & \ldots & S_{0} & \ldots & S_n\\
\vdots &  & \vdots &  & \vdots\\
S_{0} & \ldots & S_{n} & \ldots & S_{2n}\end{array}
\right),
\end{equation}
\begin{equation}
\label{f1_2_1}
\widetilde\Gamma_n = (S_{i+j+1})_{i,j=-n}^n =
\left(
\begin{array}{ccccc} S_{-2n+1} & \ldots & S_{-n+1} & \ldots & S_1\\
\vdots &  & \vdots &  & \vdots\\
S_{-n+1} & \ldots & S_{1} & \ldots & S_{n+1}\\
\vdots &  & \vdots &  & \vdots\\
S_{1} & \ldots & S_{n+1} & \ldots & S_{2n+1}\end{array}
\right),\qquad n\in \mathbb{Z}.
\end{equation}
We shall prove that conditions
\begin{equation}
\label{f1_3}
\Gamma_n\geq 0,\quad \widetilde\Gamma_n\geq 0\qquad n=0,1,2,...,
\end{equation}
are necessary and sufficient for the solvability of the moment problem~(\ref{f1_1}).

Secondly, we obtain an analytic description of all solutions of the moment problem~(\ref{f1_1}).
We shall use an abstract operator approach similar to the "pure operator"$~$approach of
Sz\"okefalvi-Nagy and Koranyi to the Nevanlinna-Pick interpolation problem,
see~\cite{cit_100_SK},\cite{cit_200_SK}.
We shall need some properties of generalized $\Pi$-resolvents of non-negative operators and
generalized $sc$-resolvents of Hermitian contractions, established by Krein and Ovcharenko
in~\cite{cit_300_Kr},\cite{cit_400_KO}.
As a by-product, we present a description of generalized $\Pi$-resolvents of a non-negative operator which
does not use improper elements or relations as it was done in the original work of Krein~\cite{cit_500_Kr}
and in the paper of Derkach and Malamud~\cite{cit_600_DM}. Here we adapt some ideas from~\cite{cit_700_Ch} of
Chumakin who studied generalized resolvents of isometric operators.

Thirdly, we obtain necessary and sufficient conditions for the strong matrix Stieltjes moment
problem to be determinate.

{\bf Notations. }
As usual, we denote by $\mathbb{R}, \mathbb{C}, \mathbb{N}, \mathbb{Z}, \mathbb{Z}_+$,
the sets of real numbers, complex numbers, positive integers, integers and non-negative integers,
respectively.
The space of $n$-dimensional complex vectors $a = (a_0,a_1,\ldots,a_{n-1})$, we denote by
$\mathbb{C}^n$, $n\in \mathbb{N}$.
If $a\in \mathbb{C}^n$, then $a^*$ means the complex conjugate vector.
By $\mathbb{P}_L$ we  denote the space of all complex Laurent polynomials, i.e. functions
$\sum_{k=a}^b \alpha_k x^k$, $a,b\in \mathbb{Z}$: $a\leq b$, $\alpha_k\in \mathbb{C}$.

\noindent
Let $M(x)$ be a left-continuous non-decreasing matrix function $M(x) = ( m_{k,l}(x) )_{k,l=0}^{N-1}$
on $\mathbb{R}$, $M(-\infty)=0$, and $\tau_M (x) := \sum_{k=0}^{N-1} m_{k,k} (x)$;
$\Psi(x) = ( dm_{k,l}/ d\tau_M )_{k,l=0}^{N-1}$.
By $L^2(M)$ we denote a set (of classes of equivalence)
of vector-valued functions $f: \mathbb{R}\rightarrow \mathbb{C}^N$,
$f = (f_0,f_1,\ldots,f_{N-1})$, such that (see, e.g.,~\cite{cit_800_M})
$$ \| f \|^2_{L^2(M)} := \int_\mathbb{R}  f(x) \Psi(x) f^*(x) d\tau_M (x) < \infty. $$
The space $L^2(M)$ is a Hilbert space with a scalar product
$$ ( f,g )_{L^2(M)} := \int_\mathbb{R}  f(x) \Psi(x) g^*(x) d\tau_M (x),\qquad f,g\in L^2(M). $$
We denote $\vec e_k = (\delta_{0,k},\delta_{1,k},...,\delta_{N-1,k})$, $0\leq k\leq N-1$,
where $\delta_{j,k}$ is the Kronecker delta.

If H is a Hilbert space then $(\cdot,\cdot)_H$ and $\| \cdot \|_H$ mean
the scalar product and the norm in $H$, respectively.
Indices may be omitted in obvious cases.
For a linear operator $A$ in $H$, we denote by $D(A)$
its  domain, by $R(A)$ its range, by $\mathop{\rm Ker}\nolimits A$
its kernel, and $A^*$ means the adjoint operator
if it exists. If $A$ is invertible then $A^{-1}$ means its
inverse. $\overline{A}$ means the closure of the operator, if the
operator is closable.
If $A$ is self-adjoint, by $R_z(A)$ we denote the resolvent of $A$, $z\in \mathbb{C}\backslash \mathbb{R}$.
If $A$ is bounded then $\| A \|$ denotes its
norm.
For an arbitrary set of elements $\{ x_n \}_{n\in \mathbb{Z}}$ in
$H$, we denote by $\mathop{\rm Lin}\nolimits\{ x_n \}_{n\in \mathbb{Z}}$
and $\mathop{\rm span}\nolimits\{ x_n \}_{n\in \mathbb{Z}}$ the linear span
and the closed linear span (in the norm of $H$), respectively.
For a set $M\subseteq H$
we denote by $\overline{M}$ the closure of $M$ in the norm of $H$.
By $E_H$ we denote the identity operator in $H$, i.e. $E_H x = x$,
$x\in H$. If $H_1$ is a subspace of $H$, then $P_{H_1} =
P_{H_1}^{H}$ is an operator of the orthogonal projection on $H_1$
in $H$. By $[H]$ we denote a set of all bounded linear operators $A$ in $H$, $D(A)=H$.

\section{The solvability of the strong matrix Stieltjes moment problem.}

In this section we are going  to establish the following theorem.

\begin{thm}
\label{t2_1}
Let the strong matrix Stieltjes moment problem~(\ref{f1_1}) with a set of moments
$\{ S_n \}_{n\in \mathbb{Z}}$ be given. The moment problem has a solution
if and only if conditions~(\ref{f1_3}) are satisfied.
\end{thm}

\begin{proof}
{\it Necessity. } Let the strong matrix Stieltjes moment problem~(\ref{f1_1}) has a solution $M(x)$.
Choose an arbitrary vector function $f(x) = \sum_{k=-n}^n \sum_{j=0}^{N-1} f_{j,k} x^k \vec e_j$,
$f_{j,k}\in \mathbb{C}$. This function belongs to $L^2(M)$ and
$$ 0 \leq \int_\mathbb{R} f(x) x^s dM(x) f^*(x) =
\sum_{k,r=-n}^n \sum_{j,l=0}^{N-1} f_{j,k} \overline{f_{l,r}}
\int_{ \mathbb{R}_+ } x^{k+r+s} \vec e_j dM(x) \vec e_l^*  $$
$$ =
\sum_{k,r=-n}^n \sum_{j,l=0}^{N-1} f_{j,k} \vec e_j S_{k+r+s} \overline{f_{l,r}} \vec e_l^*  =
\sum_{k,r=-n}^n (f_{0,k},f_{1,k},...,f_{N-1,k}) S_{k+r+s}
$$
$$ * (f_{0,r},f_{1,r},...,f_{N-1,r})^* =
\left\{ \begin{array}{cc} v \Gamma_n v^*, & s=0\\
v \widetilde\Gamma_n v^*, & s=1\end{array}\right.,
$$
where $v = (f_{0,-n},f_{1,-n},...,f_{N-1,-n},f_{0,-n+1},f_{1,-n+1},...,f_{N-1,-n+1},...,$

\noindent
$f_{0,n},f_{1,n},...,f_{N-1,n})$. Here we make use of the rules for the multiplication of block matrices.

{\it Sufficiency. }
Let the strong matrix Stieltjes moment problem~(\ref{f1_1}) be given and conditions~(\ref{f1_3})
be satisfied.
Let $S_j = (S_{j;k,l})_{k,l=0}^{N-1}$, $S_{j;k,l}\in \mathbb{C}$, $j\in \mathbb{Z}$.
Consider the following infinite block matrix:
\begin{equation}
\label{f2_1}
\Gamma = (S_{i+j})_{i,j=-\infty}^\infty =
\left(
\begin{array}{ccccccc}
 & \vdots &  & \vdots &  & \vdots &  \\
\ldots & S_{-2n} & \ldots & S_{-n} & \ldots & S_0 & \ldots\\
\ldots & \vdots &  & \vdots &  & \vdots & \ldots\\
\ldots & S_{-n} & \ldots & \mbox{ \fbox{$S_{0}$} } & \ldots & S_n & \ldots\\
\ldots & \vdots &  & \vdots &  & \vdots & \ldots\\
\ldots & S_{0} & \ldots & S_{n} & \ldots & S_{2n} & \ldots\\
  & \vdots &  & \vdots &  & \vdots &  \end{array}
\right),
\end{equation}
where the element in the box corresponds to the indices $i=j=0$.

\noindent
We assume that the left upper entry of the element in the box stands in row~0, column~0. Let us numerate
rows (columns) in the increasing order to the bottom (respectively to the right).
Then we numerate rows (columns) in the decreasing order to the top (respectively to the left).
Thus, the matrix $\Gamma$ may be viewed as a numerical matrix: $\Gamma = (\gamma_{k,l})_{k,l=-\infty}^\infty$,
$\gamma_{k,l}\in \mathbb{C}$.
Observe that the following equalities hold
\begin{equation}
\label{f2_2}
\gamma_{rN+j,tN+n} = S_{r+t;j,n},\qquad r,t\in \mathbb{Z},\ 0\leq j,n\leq N-1.
\end{equation}
From conditions~(\ref{f1_3}) it easily follows that
\begin{equation}
\label{f2_2_1}
(\gamma_{k,l})_{k,l=-r}^r \geq 0,\quad
(\gamma_{k+N,l})_{k,l=-r}^r \geq 0,\qquad \forall r\in \mathbb{Z}_+.
\end{equation}
The first inequality in the latter relation implies that there exist
a Hilbert space $H$ and a set of elements $\{ x_n \}_{n\in \mathbb{Z}}$ in $H$ such that
\begin{equation}
\label{f2_3}
(x_n,x_m)_H = \gamma_{n,m},\qquad n,m\in \mathbb{Z},
\end{equation}
and $\mathop{\rm span}\nolimits \{ x_n \}_{n\in \mathbb{Z}} = H$, see Lemma in~\cite[p. 177]{cit_200_SK}.
The latter fact is well known and goes back to a paper of Gelfand, Naimark~\cite{cit_900_GN}.

\noindent
By~(\ref{f2_2}) we get
\begin{equation}
\label{f2_4}
\gamma_{a + N,b} = \gamma_{a,b + N},\qquad a,b\in \mathbb{Z}.
\end{equation}
Set $L = \mathop{\rm Lin}\nolimits\{ x_n \}_{n\in \mathbb{Z}}$. Choose an arbitrary element
$x\in L$.
Let $x = \sum_{k=-\infty}^\infty \alpha_k x_k$, $x = \sum_{k=-\infty}^\infty \beta_k x_k$,
where $\alpha_k,\beta_k\in \mathbb{C}$. Here only a finite number of coefficients $\alpha_k$, $\beta_k$
are non-zero. {\it In what follows, this will be assumed in analogous situations with elements of the linear span.}
By~(\ref{f2_3}),(\ref{f2_4}) we may write
$$ \left( \sum_{k=-\infty}^\infty \alpha_k x_{k+ N}, x_l \right) =
\sum_{k=-\infty}^\infty \alpha_k \gamma_{k+ N,l} = \sum_{k=-\infty}^\infty \alpha_k \gamma_{k,l+ N} = $$
$$ = \sum_{k=-\infty}^\infty \alpha_k ( x_{k}, x_{l+ N} ) =
\left( \sum_{k=-\infty}^\infty \alpha_k x_{k}, x_{l+ N} \right) = (x,x_{l+ N}),\qquad l\in \mathbb{Z}. $$
Similarly we conclude that
$\left( \sum_{k=-\infty}^\infty \beta_k x_{k+ N}, x_l \right) = (x,x_{l+ N})$, $l\in \mathbb{Z}$.
Since $\overline{L} = H$, we get
$\sum_{k=-\infty}^\infty \alpha_k x_{k+ N} = \sum_{k=-\infty}^\infty \beta_k x_{k+ N}$.

\noindent
Set
\begin{equation}
\label{f2_6}
A_0 x = \sum_{k=0}^\infty \alpha_k x_{k+N},\qquad x\in L,\ x = \sum_{k=-\infty}^\infty \alpha_k x_{k},\
\alpha_k\in \mathbb{C}.
\end{equation}
The above considerations ensure us that the operator $A_0$ is defined correctly.
Choose arbitrary elements $x,y\in L$, $x = \sum_{k=-\infty}^\infty \alpha_k x_{k}$,
$y = \sum_{n=-\infty}^\infty \beta_n x_{n}$, $\alpha_k,\beta_n\in \mathbb{C}$,
and write
$$ (A_0 x,y)_H = \left( \sum_{k=-\infty}^\infty \alpha_k x_{k+N},\sum_{n=-\infty}^\infty \beta_n x_{n} \right)_H =
\sum_{k,n=-\infty}^\infty \alpha_k \overline{\beta_n} (x_{k+N},x_n)_H = $$
$$ = \sum_{k,n=-\infty}^\infty \alpha_k \overline{\beta_n} (x_{k},x_{n+N})_H
= \left( \sum_{k=-\infty}^\infty \alpha_k x_{k},\sum_{n=-\infty}^\infty \beta_n x_{n+N} \right)_H =
(x,A_0 y)_H. $$
Moreover, we have
\begin{equation}
\label{f2_6_1}
(A_0 x,x)_H = \sum_{k,n=-\infty}^\infty \alpha_k \overline{\beta_n} (x_{k+N},x_n)_H =
\sum_{k,n=-\infty}^\infty \alpha_k \overline{\beta_n} \gamma_{k+N,n} \geq 0.
\end{equation}
Thus, the operator $A_0$ is a non-negative symmetric operator in $H$. Set $A = \overline{A_0}$.
The operator $A$ always has a non-negative self-adjoint extension $\widetilde A$
in a Hilbert space $\widetilde H\supseteq H$~\cite[Theorem 7, p.450]{cit_300_Kr} .
We may assume that $\mathop{\rm Ker}\nolimits \widetilde A = \{ 0 \}$.
In the opposite case, since $\mathop{\rm Ker}\nolimits \widetilde A \perp R(\widetilde A)$,
$R(\widetilde A) \supseteq L$,
we conclude that $\mathop{\rm Ker}\nolimits \widetilde A \perp H$. Therefore the operator $\widetilde A$,
restricted to $\widetilde H \ominus \mathop{\rm Ker}\nolimits \widetilde A$, also will be
a self-adjoint extension of the operator
$A$, with a null kernel.

\noindent
Let $\{ \widetilde E_\lambda \}_{\lambda\in \mathbb{R}}$ be the left-continuous
orthogonal resolution of unity of the operator $\widetilde A$.
By the induction argument it is easy to check that
$$ x_{rN+j} = A^r x_j,\qquad r\in \mathbb{Z},\ 0\leq j\leq N-1. $$
By~(\ref{f2_2}),(\ref{f2_3}) we may write
$$ S_{r;j,n} = \gamma_{rN+j,n} = ( x_{rN+j},x_{n} )_H = (A^r x_j, x_n)_H =
( \widetilde A^r x_j, x_n)_{\widetilde H} $$
$$  =
\int_{\mathbb{R}_+} \lambda^{r} d (\widetilde E_\lambda x_j, x_n)_{\widetilde H} =
\int_{\mathbb{R}_+} \lambda^{r} d \left( P^{\widetilde H}_H \widetilde E_\lambda x_j, x_n \right)_{H},\quad
0\leq n\leq N-1. $$
Therefore we get
\begin{equation}
\label{f2_9}
S_{r} = \int_{\mathbb{R}_+} \lambda^{r} d \widetilde M(\lambda),\qquad r\in \mathbb{Z},
\end{equation}
where $\widetilde M(\lambda) := \left( \left( P^{\widetilde H}_H \widetilde E_\lambda x_j,
x_n \right)_{H} \right)_{j,n=0}^{N-1}$.
Therefore the matrix function
$\widetilde M(\lambda)$ is a solution of the moment problem~(\ref{f1_1})
(From the properties of  the orthogonal resolution of unity it easily follows that
$\widetilde M (\lambda)$ is left-continuous on $(0,+\infty)$, non-decreasing and
$\widetilde M(0) = 0$).

\end{proof}

\section{An analytic description of solutions of the strong matrix Stieltjes moment problem.}

Let $A$ be an arbitrary closed Hermitian operator in a Hilbert space $H$, $D(A) \subseteq H$.
Let $\widehat A$ be an arbitrary self-adjoint extension of $A$
in a Hilbert space $\widehat H\supseteq H$.
Denote by $\{ \widehat E_\lambda\}_{\lambda\in \mathbb{R}}$
its orthogonal resolution of unity. Recall that an operator-valued function
$\mathbf R_z = P_H^{\widehat H} R_z(\widehat A)$ is said to be a {\bf generalized resolvent} of $A$,
$z\in \mathbb{C}\backslash \mathbb{R}$.
A function
$\mathbf E_\lambda = P_H^{\widehat H} \widehat E_\lambda$, $\lambda\in \mathbb{R}$,
is said to be a {\bf spectral function} of $A$.
There exists a bijective correspondence between generalized resolvents and
left-continuous (or normalized in some other way) spectral functions established
by the following relation~(\cite{c_1000_AG}):
\begin{equation}
\label{f3_1}
(\mathbf R_z f,g)_H = \int_\mathbb{R} \frac{1}{\lambda - z}
d( \mathbf E_\lambda f,g)_H,\qquad f,g\in H,\ z\in \mathbb{C}\backslash \mathbb{R}.
\end{equation}

If the operator $A$ is densely defined symmetric and non-negative ($A\geq 0$), and
the extension $\widehat A$ is self-adjoint and non-negative, then the corresponding
generalized resolvent $\mathbf{R}_z$ and the spectral function $\mathbf{E}_\lambda$ are said to be
a {\bf generalized $\Pi$-resolvent} and a {\bf $\Pi$-spectral function} of $A$. Relation~(\ref{f3_1})
establishes a bijective correspondence between generalized $\Pi$-resolvents and
left-continuous $\Pi$-spectral functions.

If the operator $A$ is a Hermitian contraction ($\| A \| \leq 1$), and
the extension $\widehat A$ is a self-adjoint contraction, then the corresponding
generalized resolvent $\mathbf{R}_z$ and the spectral function $\mathbf{E}_\lambda$ are said to be
a {\bf generalized $sc$-resolvent} and a {\bf $sc$-spectral function} of $A$. Relation~(\ref{f3_1})
establishes a bijective correspondence between generalized $sc$-resolvents and
left-continuous $sc$-spectral functions, as well.

If a generalized $\Pi$-resolvent (a generalized $sc$-resolvent) is generated by an extension {\it inside $H$},
i.e. $\widehat H = H$, then it is said to be a {\bf canonical $\Pi$-resolvent} (respectively
a {\bf canonical $sc$-resolvent}).

Firstly, we shall obtain a description of solutions of the strong matrix Stieltjes moment problem
by virtue of $\Pi$-spectral functions.

\begin{thm}
\label{t3_1}
Let the strong matrix Stieltjes moment problem~(\ref{f1_1}) be given and condition~(\ref{f1_3}) be satisfied.
Suppose that the operator $A=\overline{A_0}$ in a Hilbert space $H$ is constructed for the moment problem
by~(\ref{f2_6}) and the preceding procedure.
All solutions of the moment problem have the following form
\begin{equation}
\label{f3_2}
M(\lambda) = (m_{k,j} (\lambda))_{k,j=0}^{N-1},\quad
m_{k,j} (\lambda) = ( \mathbf E_\lambda x_k, x_j)_H,
\end{equation}
where $\mathbf E_\lambda$ is a left-continuous $\Pi$-spectral function of the operator $A$.
On the other hand, each left-continuous $\Pi$-spectral function of the operator $A$ generates
by~(\ref{f3_2}) a solution of the moment problem.
Moreover, the correspondence between all left-continuous $\Pi$-spectral functions of the operator $A$
and all solutions of the moment problem is bijective.
\end{thm}
\begin{proof}
Let $\mathbf E_\lambda$ be an arbitrary $\Pi$-spectral function of the operator $A$.
It corresponds to a self-adjoint operator $\widetilde A\supseteq A$ in a Hilbert space $\widetilde H
\supseteq H$. Then we repeat considerations after~(\ref{f2_6_1})
to obtain that $M(\lambda)$, given by~(\ref{f3_2}), is a solution of the moment problem~(\ref{f1_1}).

Let $\widehat M(x) = ( \widehat m_{k,l}(x) )_{k,l=0}^{N-1}$ be an arbitrary solution
of the moment problem~(\ref{f1_1}).
Consider the space $L^2(\widehat M)$ and let
$Q$ be the operator of  multiplication by an independent variable in $L^2(\widehat M)$.
A set (of classes of equivalence) of functions
$f\in L^2(\widehat M)$ such that
(the corresponding class includes) $f=(f_0,f_1,\ldots, f_{N-1})$, $f\in \mathbb{P}_L$, we denote by
$\mathbb{P}^2_L(\widehat M)$. Set $L^2_{L}(\widehat M) := \overline{ \mathbb{P}^2_L(\widehat M) }$.

For an arbitrary vector Laurent polynomial $f=(f_0,f_1,\ldots, f_{N-1})$, $f_j\in \mathbb{P}_L$,
there exists a unique representation of the following form:
\begin{equation}
\label{f3_3}
f(x) = \sum_{k=0}^{N-1} \sum_{j=-\infty}^\infty \alpha_{k,j} x^j \vec e_k,\qquad \alpha_{k,j}\in \mathbb{C},
\end{equation}
where all but finite number of coefficients $\alpha_{k,j}$ are zero.
Choose another vector Laurent polynomial $g$ with a representation
\begin{equation}
\label{f3_4}
g(x) = \sum_{l=0}^{N-1} \sum_{r=-\infty}^\infty \beta_{l,r} x^r \vec e_l,\qquad \beta_{l,r}\in \mathbb{C}.
\end{equation}
We may write
$$ (f,g)_{L^2(\widehat M)} = \sum_{k,l=0}^{N-1} \sum_{j,r=-\infty}^\infty \alpha_{k,j}\overline{\beta_{l,r}}
\int_{\mathbb{R}_+} x^{j+r} \vec e_k d\widehat M(x) \vec e_l^* $$
\begin{equation}
\label{f3_5}
= \sum_{k,l=0}^{N-1}
\sum_{j,r=-\infty}^\infty \alpha_{k,j}\overline{\beta_{l,r}}
\int_{\mathbb{R}_+} x^{j+r} d\widehat m_{k,l}(x)
= \sum_{k,l=0}^{N-1} \sum_{j,r=-\infty}^\infty \alpha_{k,j}\overline{\beta_{l,r}}
S_{j+r;k,l}.
\end{equation}
On the other hand, we have
$$ \left( \sum_{j=-\infty}^\infty \sum_{k=0}^{N-1} \alpha_{k,j} x_{jN+k},
\sum_{r=-\infty}^\infty \sum_{l=0}^{N-1} \beta_{l,r} x_{rN+l} \right)_H =
\sum_{k,l=0}^{N-1} \sum_{j,r=-\infty}^\infty \alpha_{k,j}\overline{\beta_{l,r}} $$
\begin{equation}
\label{f3_6}
* (x_{jN+k}, x_{rN+l})_H = \sum_{k,l=0}^{N-1} \sum_{j,r=-\infty}^\infty \alpha_{k,j}\overline{\beta_{l,r}}
\gamma_{jN+k,rN+l}
= \sum_{k,l=0}^{N-1} \sum_{j,r=-\infty}^\infty \alpha_{k,j}\overline{\beta_{l,r}}
S_{j+r;k,l}.
\end{equation}
By~(\ref{f3_5}),(\ref{f3_6}) we get
\begin{equation}
\label{f3_7}
(f,g)_{L^2(\widehat M)} = \left( \sum_{j=-\infty}^\infty \sum_{k=0}^{N-1} \alpha_{k,j} x_{jN+k},
\sum_{r=-\infty}^\infty \sum_{l=0}^{N-1} \beta_{l,r} x_{rN+l} \right)_H.
\end{equation}
Set
\begin{equation}
\label{f3_8}
Vf = \sum_{j=-\infty}^\infty \sum_{k=0}^{N-1} \alpha_{k,j} x_{jN+k},
\end{equation}
for a vector Laurent polynomial $f(x) = \sum_{k=0}^{N-1} \sum_{j=-\infty}^\infty \alpha_{k,j} x^j \vec e_k$.
If $f$, $g$ are vector Laurent polynomials with representations~(\ref{f3_3}),(\ref{f3_4}), such that
$\| f-g \|_{L^2(\widehat M)} = 0$, then
by~(\ref{f3_7}) we may write
$$ \| Vf - Vg \|_H^2 = (V(f-g),V(f-g))_H = ( f-g,f-g )_{L^2(\widehat M)} = \| f-g\|_{L^2(\widehat M)}^2 = 0.$$
Thus, $V$ is correctly defined as an operator from $\mathbb{P}^2(\widehat M)$ to $H$.
Relation~(\ref{f3_7}) shows that $V$  is an isometric transformation from
$\mathbb{P}^2_L(\widehat M)$ on $L$.
We extend $V$ by continuity to an isometric transformation from $L^2_L(\widehat M)$ on $H$.
Observe that
\begin{equation}
\label{f3_9}
V x^j \vec e_k = x_{jN+k},\qquad j\in \mathbb{Z};\quad 0\leq k\leq N-1.
\end{equation}
Let $L^2_1 (\widehat M) := L^2(\widehat M)\ominus L^2_L (\widehat M)$, and
$U := V\oplus E_{L^2_1 (\widehat M)}$. The operator $U$ is an isometric transformation from
$L^2(\widehat M)$ on $H\oplus L^2_1 (\widehat M)=:\widehat H$.
Set
$$ \widehat A := UQU^{-1}. $$
The operator $\widehat A$ is a self-adjoint operator in $\widehat H$.
Let $\{ \widehat E_\lambda \}_{\lambda\in \mathbb{R}}$ be its left-continuous orthogonal resolution of unity.
Notice that
$$ UQU^{-1} x_{jN+k} = VQV^{-1} x_{jN+k} = VQ x^j \vec e_k = V x^{j+1} \vec e_k =
x_{(j+1)N+k} = $$
$$= x_{jN+k+N} = Ax_{jN+k},\qquad j\in \mathbb{Z};\quad 0\leq k\leq N-1. $$
By linearity we get: $UQU^{-1} x = Ax$, $x\in L$,
and therefore $\widehat A\supseteq A$.
Choose an arbitrary $z\in \mathbb{C}\backslash \mathbb{R}$ and write
$$ \int_{\mathbb{R}_+} \frac{1}{\lambda - z} d( \widehat E_\lambda x_k, x_j)_{\widehat H} =
\left( \int_{\mathbb{R}_+} \frac{1}{\lambda - z} d\widehat E_\lambda x_k, x_j \right)_{\widehat H} $$
$$ = \left( U^{-1} \int_{\mathbb{R}_+} \frac{1}{\lambda - z} d\widehat E_\lambda x_k, U^{-1} x_j
\right)_{L^2(\widehat M)} $$
$$ = \left( \int_{\mathbb{R}_+} \frac{1}{\lambda - z} d U^{-1} \widehat E_\lambda U \vec e_k,
\vec e_j \right)_{L^2(\widehat M)} =
\left( \int_{\mathbb{R}_+} \frac{1}{\lambda - z} d E_{\lambda} \vec e_k, \vec e_j \right)_{L^2(\widehat M)} $$
$$ = \left( (Q-z)^{-1} \vec e_k, \vec e_j \right)_{L^2(\widehat M)} =
\int_{\mathbb{R}_+} \frac{1}{\lambda - z} \vec e_k d\widehat M(\lambda) \vec e_j
= \int_{\mathbb{R}_+} \frac{1}{\lambda - z} d\widehat m_{k,j}(\lambda), $$
where $E_\lambda$ is a left-continuous orthogonal resolution of unity of
the operators $Q$.
By the Stieltjes-Perron inversion formula (e.g.~\cite{cit_2000_Akh}) we conclude that
$$ \widehat m_{k,j} (\lambda) = ( P^{\widehat H}_H \widehat E_\lambda x_k, x_j)_H,\qquad \lambda\in \mathbb{R}. $$
Thus, $\widehat M$ is generated by a $\Pi$-spectral function of $A$.

Let us check that an arbitrary element $u\in L$ can be represented in the following form
\begin{equation}
\label{f3_10}
u = u_z + u_0,\qquad u_z\in H_z,\ u_0\in L_N,
\end{equation}
where $L_N := \mathop{\rm Lin}\nolimits\{ x_n \}_{n=0}^{N-1}$,
$H_z := (A-zE_H)D(A)$.
Let $u=\sum_{k=-\infty}^\infty c_k x_k$, $c_k\in \mathbb{C}$, and choose a number
$z\in \mathbb{C}\backslash \mathbb{R}$.
Suppose that $c_k=0$, if $k\leq r$ or $k\geq R$, where $r\leq -2$; $R\geq N+1$.
Set
$d_k := 0$, if $k\leq r$ or $k\geq R - N$.
Then we set
$$ d_k := \frac{1}{z}( d_{k-N} - c_k),\qquad k=r+1,...,-1; $$
$$ d_{k-N} := zd_k + c_k,\qquad k=R -1,R -2,...,N. $$
Set $v := \sum_{k=-\infty}^\infty d_k x_k \in L$.
Then we directly calculate that
$$ (A-zE_H) v - u = \sum_{k=0}^{N-1} (d_{k-N}-zd_k - c_k) x_k, $$
and relation~(\ref{f3_10}) holds.
From the latter equality it easily follows that the deficiency index of $A$
is equal to $(n,n)$, $0\leq n\leq N$.

Let us check that different left-continuous $\Pi$-spectral functions of the operator $A$ generate
different solutions of the moment problem~(\ref{f1_1}).
Suppose that two different left-continuous $\Pi$-spectral functions generate the same solution of the moment
problem.
This  means that there exist two self-adjoint operators
$A_j\supseteq A$, in Hilbert spaces $H_j\supseteq H$, such that
$P_{H}^{H_1} E_{1,\lambda} \not= P_{H}^{H_2} E_{2,\lambda}$,
and
$$ (P_{H}^{H_1} E_{1,\lambda} x_k,x_j)_H =
(P_{H}^{H_2} E_{2,\lambda} x_k,x_j)_H,\qquad 0\leq k,j\leq N-1,\quad
\lambda\in\mathbb{R}, $$
where $\{ E_{n,\lambda} \}_{\lambda\in\mathbb{R}}$ are orthogonal left-continuous
resolutions of unity of operators $A_n$, $n=1,2$.
By the linearity we get
\begin{equation}
\label{f3_11}
(P_{H}^{H_1} E_{1,\lambda} x,y)_H = (P_{H}^{H_2} E_{2,\lambda} x,y)_H,\qquad x,y\in L_N,\quad \lambda\in\mathbb{R}.
\end{equation}
Set $\mathbf R_{n,\lambda} := P_{H}^{H_n} R_{\lambda}(A_n)$, $n=1,2$.
By~(\ref{f3_11}),(\ref{f3_1}) we get
\begin{equation}
\label{f3_12}
(\mathbf R_{1,\lambda} x,y)_H =
(\mathbf R_{2,\lambda} x,y)_H,\qquad x,y\in L_N,\quad \lambda\in\mathbb{C}\backslash\mathbb{R}.
\end{equation}
Since
$$ R_{z}(A_j) (A-zE_H) x = (A_j - z E_{H_j} )^{-1} (A_j - z E_{H_j}) x = x,\qquad x\in L=D(A_0),$$
then $R_{z}(A_1) u = R_{z}(A_2) u \in H$, $u\in H_z$;
\begin{equation}
\label{f3_13}
\mathbf R_{1,z} u = \mathbf R_{2,z} u,\qquad u\in H_z,\ z\in\mathbb{C}\backslash\mathbb{R}.
\end{equation}
We may write
$$ (\mathbf R_{n,z} x, u)_H = (R_{z}(A_n) x, u)_{H_n} = ( x, R_{\overline{z}}(A_n) u)_{H_n} =
( x, \mathbf R_{n,\overline{z}} u)_H, $$
where $x\in L_N$, $u\in H_{\overline z}$, $n=1,2$,
and therefore
\begin{equation}
\label{f3_14}
(\mathbf R_{1,z} x,u)_H = (\mathbf R_{2,z} x,u)_H,\qquad x\in L_N,\ u\in H_{\overline z}.
\end{equation}
By~(\ref{f3_10}) an arbitrary element
$y\in L$ can be represented in the following form $y=y_{ \overline{z} } + y'$,
$y_{ \overline{z} }\in H_{ \overline{z} }$, $y'\in L_N$.
Using~(\ref{f3_12}) and~(\ref{f3_14}) we obtain
$$ (\mathbf R_{1,z} x,y)_H = (\mathbf R_{1,z} x, y_{ \overline{z} } + y')_H =
(\mathbf R_{2,z} x, y_{ \overline{z} } + y')_H = (\mathbf R_{2,z} x,y)_H, $$
where $x\in L_N,\ y\in L$.
Since $\overline{L}=H$, we obtain
\begin{equation}
\label{f3_15}
\mathbf R_{1,z} x = \mathbf R_{2,z} x,\qquad x\in L_N,\ z\in\mathbb{C}\backslash\mathbb{R}.
\end{equation}
For arbitrary $x\in L$, $x=x_z + x'$, $x_z\in H_z$, $x'\in L_N$, using relations~(\ref{f3_13}),(\ref{f3_15})
we get
$$ \mathbf R_{1,z} x = \mathbf R_{1,z} (x_z + x') =
\mathbf R_{2,z} (x_z + x') = \mathbf R_{2,z} x,\qquad x\in L,\ z\in\mathbb{C}\backslash\mathbb{R}, $$
and therefore $\mathbf R_{1,z} = \mathbf R_{2,z}$, $z\in\mathbb{C}\backslash\mathbb{R}$.
By~(\ref{f3_1}) this means that the corresponding $\Pi$-spectral functions coincide.
The obtained contradiction completes the proof.
\end{proof}

We shall use some known important facts about sc-resolvents, see~\cite{cit_400_KO}.
Let $B$ be an arbitrary Hermitian contraction in a Hilbert space $H$.
Set $\mathcal D=D(B)$, $\mathcal R = H \ominus \mathcal D$.
A set of all self-adjoint contractive extensions of $B$ inside $H$, we denote by $\mathcal{B}_H(B)$.
A set of all self-adjoint contractive extensions of $B$ in a Hilbert space $\widetilde H\supseteq H$,
we denote by $\mathcal{B}_{\widetilde H} (B)$.
By Krein's theorem~\cite[Theorem 2, p. 440]{cit_300_Kr}, there always exists
a self-adjoint extension $\widehat B$ of the operator $B$ in $H$ with the norm $\| B \|$.
Therefore the set $\mathcal{B}_H(B)$ is non-empty.
There are the "minimal"$~$ element $B^\mu$ and
the "maximal"$~$ element $B^M$ in this set, such that $\mathcal{B}_H(B)$ coincides with the operator
segment
\begin{equation}
\label{f3_16}
B^\mu \leq \widetilde B\leq B^M.
\end{equation}
In the case $B^\mu = B^M$ the set $\mathcal{B}_H(B)$ consists of a unique element. This case is said to be
{\bf determinate}.
The case $B^\mu \not= B^M$  is called {\bf indeterminate}.
The case $B^\mu x \not= B^M x$, $x\in \mathcal{R}\backslash\{ 0 \}$, is said to be
{\bf completely indeterminate}.
The indeterminate case can be always reduced to the completely indeterminate. If
$\mathcal{R}_0 = \{ x\in \mathcal{R}:\ B^\mu x = B^M x\}$, we may set
\begin{equation}
\label{f3_17}
B_e x = B x,\ x\in \mathcal{D};\quad B_e x = B^\mu x,\ x\in \mathcal{R}_0.
\end{equation}
The sets of generalized sc-resolvents for $B$ and for $B_e$ coincide (\cite[p. 1039]{cit_400_KO}).

\noindent
Elements of $\mathcal{B}_H(B)$ are {\bf canonical} (i.e. inside $H$) extensions of $B$ and their
resolvents are said to be {\bf canonical sc-resolvents} of $B$.
On the other hand, elements  of $\mathcal{B}_{\widetilde H} (B)$ for all possible $\widetilde H\supseteq H$
generate generalized sc-resolvents of $B$ (here the space $\widetilde H$ is not
fixed). The set of all generalized sc-resolvents we denote by
$\mathcal{R}^c(B)$.
Set
\begin{equation}
\label{f3_18}
C = B^M - B^\mu,
\end{equation}
\begin{equation}
\label{f3_19}
Q_\mu (z) = \left.\left( C^{\frac{1}{2}} R^\mu_z C^{\frac{1}{2}} + E_H \right)\right|_{\mathcal{R}},\qquad
z\in\mathbb{C}\backslash[-1,1],
\end{equation}
where $R^\mu_z = (B^\mu - zE_H)^{-1}$.

\noindent
An operator-valued function $k(z)$ with values in $[\mathcal{R}]$ belongs to the class $R_{\mathcal{R}}[-1,1]$ if

\begin{itemize}

\item[1)] $k(z)$ is analytic in $z\in\mathbb{C}\backslash[-1,1]$ and
$$ \frac{ \mathop{\rm Im}\nolimits k(z) }{ \mathop{\rm Im}\nolimits z }\leq 0,\qquad z\in\mathbb{C}:\ \mathop{\rm Im}\nolimits z\not= 0; $$

\item[2)] For $z\in \mathbb{R}\backslash[-1,1]$, $k(z)$ is a self-adjoint non-negative contraction.

\end{itemize}
Notice that functions from the class $R_{\mathcal{R}}[-1,1]$ admit a special integral
representation, see~\cite{cit_400_KO}.

\begin{thm}
\label{t3_2} (\cite[p. 1053]{cit_400_KO}).
Let $B$ be a Hermitian contraction in a Hilbert space $H$ with $D(B)=\mathcal{D}$; $\mathcal{R} =
H\ominus \mathcal{D}$. Suppose that for $B$ it takes place the completely indeterminate case and
the corresponding operator $C$, as an operator in $\mathcal{R}$, has an inverse in $[\mathcal{R}]$.
Then the following equality:
\begin{equation}
\label{f3_20}
\widetilde R_z^c = R^\mu_z - R^\mu_z C^{\frac{1}{2}} k(z) \left(
E_\mathcal{R} + (Q_\mu(z)-E_\mathcal{R}) k(z)
\right)^{-1}
C^{\frac{1}{2}} R^\mu_z,
\end{equation}
where $k(z)\in R_{\mathcal{R}}[-1,1]$, $\widetilde R_z^c\in \mathcal{R}^c(B)$,
establishes a bijective correspondence between the set $R_{\mathcal{R}}[-1,1]$ and the set $\mathcal{R}^c(B)$.

Moreover, the canonical resolvents correspond in~(\ref{f3_20}) to the constant functions
$k(z)\equiv K$, $K\in [0,E_{\mathcal{R}}]$.
\end{thm}

Let $A$ be an arbitrary non-negative symmetric operator in a Hilbert space $H$, $\overline{D(A)} = H$.
We are going to obtain a formula for the generalized $\Pi$-resolvents of $A$, by
virtue of Theorem~\ref{t3_2}. Set
\begin{equation}
\label{f3_21}
T = (E_H - A)(E_H + A)^{-1} = -E_H + 2(E_H+A)^{-1},\qquad D(T) = (A+E_H) D(A).
\end{equation}
Then
\begin{equation}
\label{f3_22}
A = (E_H - T)(E_H + T)^{-1} = -E_H + 2(E_H+T)^{-1},\qquad D(A) = (T+E_H) D(T).
\end{equation}
The latter transformations were introduced and intensively studied by Krein~\cite{cit_300_Kr}.
The  operator $T$ is a Hermitian contraction in $H$. In fact, for an arbitrary
$h = (A+E_H)f$, $f\in D(A)$ we may write
$$ \| Th \|_H^2 = \| (-E_H + 2(E_H+A)^{-1})(A+E_H)f \|_H^2 =
\| -Af+f \|_H^2 $$
$$ = \| Af \|_H^2 + \|f\|_H^2 - 2(Af,f)_H \leq \| Af \|_H^2 + \|f\|_H^2 + 2(Af,f)_H = \| h \|_H^2. $$
Let $\widetilde A\supseteq A$ be a non-negative self-adjoint extension of $A$ in a Hilbert space
$\widetilde H\supseteq H$. Then the operator
\begin{equation}
\label{f3_23}
\widetilde T = (E_{\widetilde H} - \widetilde A)(E_{\widetilde H} + \widetilde A)^{-1} =
-E_{\widetilde H} + 2(E_{\widetilde H}+ \widetilde A)^{-1},\qquad D(\widetilde T) = (\widetilde A+E_{\widetilde H})
D(\widetilde A),
\end{equation}
is a self-adjoint contraction $\widetilde T\supseteq T$ in $\widetilde H$, and
\begin{equation}
\label{f3_24}
\widetilde A = (E_{\widetilde H} - \widetilde T)(E_{\widetilde H} + \widetilde T)^{-1} =
-E_{\widetilde H} + 2(E_{\widetilde H}+ \widetilde T)^{-1},\qquad D(\widetilde A) = (\widetilde T+E_{\widetilde H})
D(\widetilde T).
\end{equation}
Consider the following fractional linear transformation:
\begin{equation}
\label{f3_25}
z = \frac{1-\lambda}{1+\lambda} = -1 + 2\frac{1}{1+\lambda};\quad
\lambda = \frac{1-z}{1+z} = -1 + 2\frac{1}{1+z}.
\end{equation}
Choose an arbitrary $z\in \mathbb{C}\backslash \mathbb{R}$ and set
$\lambda := \frac{1-z}{1+z}$. Observe that $\lambda\in \mathbb{C}\backslash \mathbb{R}$. Then
$$ R_z(\widetilde T) = (\widetilde T - z E_{\widetilde H})^{-1}
= \left(
-E_{\widetilde H} + 2(E_{\widetilde H} + \widetilde A)^{-1} - \frac{1-\lambda}{1+\lambda} E_{\widetilde H}
\right)^{-1} $$
$$ = \left(
\frac{(-2)}{ 1+\lambda } (E_{\widetilde H} + \widetilde A)(E_{\widetilde H} + \widetilde A)^{-1} +
2(E_{\widetilde H} + \widetilde A)^{-1}
\right)^{-1} $$
$$ = \left(
\left(
\frac{2\lambda}{ 1+\lambda } E_{\widetilde H} - \frac{2}{1+\lambda} \widetilde A
\right)
(E_{\widetilde H} + \widetilde A)^{-1}
\right)^{-1} $$
$$ = -\frac{\lambda+1}{2}
\left(
\left(
\widetilde A - \lambda E_{\widetilde H}
\right)
(E_{\widetilde H} + \widetilde A)^{-1}
\right)^{-1} =
-\frac{\lambda+1}{2} (E_{\widetilde H} + \widetilde A) (\widetilde A - \lambda E_{\widetilde H})^{-1}
$$
$$ = -\frac{ (\lambda+1)^2 }{2} (\widetilde A - \lambda E_{\widetilde H})^{-1} -
\frac{\lambda + 1}{2} E_{\widetilde H}
= -\frac{ (\lambda+1)^2 }{2} R_\lambda (\widetilde A) - \frac{\lambda + 1}{2} E_{\widetilde H}. $$
Therefore
\begin{equation}
\label{f3_26}
R_\lambda(\widetilde A) = -\frac{2}{(\lambda + 1)^2} R_{\frac{1-\lambda}{1+\lambda}}(\widetilde T) -
\frac{1}{\lambda + 1} E_{\widetilde H},\qquad \forall\lambda\in \mathbb{C}\backslash \mathbb{R}.
\end{equation}
Applying the orthogonal projection on $H$, we get
\begin{equation}
\label{f3_27}
\mathbf{R}_\lambda(A) = -\frac{2}{(\lambda + 1)^2} \mathbf{R}_{\frac{1-\lambda}{1+\lambda}}(T) -
\frac{1}{\lambda + 1} E_{H},\qquad \forall\lambda\in \mathbb{C}\backslash \mathbb{R}.
\end{equation}
Here $\mathbf{R}_\lambda(A)$ is the generalized $\Pi$-resolvent corresponding to $\widetilde A$,
and $\mathbf{R}_z(T)$ is the generalized $sc$-resolvent corresponding to $\widetilde T$.
Thus, an arbitrary generalized $\Pi$-resolvent of $A$ can be constructed by a generalized
$sc$-resolvent of $T$ by relation~(\ref{f3_27}).

\noindent
On the other hand, choose an arbitrary $sc$-resolvent $\mathbf{R}_z'(T)$ of $T$. It corresponds
to a self-adjoint contractive extension $\widehat T\supseteq T$ in a  Hilbert space $\widehat H
\supseteq H$. Observe that
$$ \mathop{\rm Ker}\nolimits (E_{\widehat H}+ \widehat T)\perp
R(E_{\widehat H}+ \widehat T)\supseteq R(E_{H}+ T) = D(A), $$
and therefore $\mathop{\rm Ker}\nolimits (E_{\widehat H}+ \widehat T) \perp H$.
We may assume that $H_1 := \mathop{\rm Ker}\nolimits (E_{\widehat H}+ \widehat T) = \{ 0 \}$, since in the
opposite case one may consider the operator $\widehat T$ restricted to $\widehat H\ominus H_1\supseteq H$.
Then we set
\begin{equation}
\label{f3_28}
\widehat A = (E_{\widehat H} - \widehat T)(E_{\widehat H} + \widehat T)^{-1} =
-E_{\widehat H} + 2(E_{\widehat H}+ \widehat T)^{-1},\qquad D(\widehat A) = (\widehat T+E_{\widehat H})
D(\widehat T).
\end{equation}
The operator $\widehat A$ is densely defined since $\widehat A\supseteq A$,
and it is self-adjoint.
For an arbitrary $u\in D(\widehat T)$ we may write
$$ (\widehat A (\widehat T + E_{\widehat H}) u, (\widehat T + E_{\widehat H}) u)_{\widehat H} =
(-\widehat T u + u, \widehat T u + u)_{\widehat H} =
\| u \|_{\widehat H}^2 - \| \widehat T u \|_{\widehat H}^2
\geq 0. $$
Thus, the operator $\widehat A$ is non-negative.
Observe that
\begin{equation}
\label{f3_29}
\widehat T = (E_{\widehat H} - \widehat A)(E_{\widehat H} + \widehat A)^{-1} =
-E_{\widehat H} + 2(E_{\widehat H} + \widehat A)^{-1}.
\end{equation}
Repeating the considerations after relation~(\ref{f3_25}), we obtain that
\begin{equation}
\label{f3_30}
\mathbf{R}_\lambda'(A) = -\frac{2}{(\lambda + 1)^2} \mathbf{R}_{\frac{1-\lambda}{1+\lambda}}'(T) -
\frac{1}{\lambda + 1} E_{H},\qquad \forall\lambda\in \mathbb{C}\backslash \mathbb{R},
\end{equation}
gives a generalized $\Pi$-resolvent of $A$ (corresponding to $\widehat A$).

\noindent
Consequently, the relation~(\ref{f3_27}) establishes a bijective correspondence between the set
of all $sc$-resolvents of $T$ and the set of all $\Pi$-resolvents of $A$.
It is not hard to see that the canonical $sc$-resolvents are related to the canonical $\Pi$-resolvents.

They say that for the operator $A$ it takes place a {\bf completely indeterminate case},
if for the corresponding operator $T$ it takes place the completely indeterminate case~\cite{cit_3000_KO}.

\noindent
It is known that all self-adjoint contractive extensions of $T$ are extensions of
the extended operator $T_e$ defined by~(\ref{f3_17})~\cite[Theorem 1.4]{cit_400_KO}.
Set
\begin{equation}
\label{f3_31}
A_e = (E_{H} - T_e)(E_{H} + T_e)^{-1} =
-E_{H} + 2(E_{H}+ T_e)^{-1},\ D(A_e) = (T_e+E_{H})D(T_e).
\end{equation}
It is easily seen that the above operator $\widetilde A$ is an extension of $A_e$.
Therefore the sets of generalized $\Pi$-resolvents for $A$ and for $A_e$ coincide.

\begin{thm}
\label{t3_3}
Let $A$ be a non-negative symmetric operator in a Hilbert space $H$, $\overline{D(A)} = H$.
Suppose that for $A$ it takes place the completely indeterminate case.
Let $T$ be given by~(\ref{f3_21}); $\mathcal{D} = D(T)$, $\mathcal{R} = H\ominus \mathcal{D}$.
Suppose that the corresponding operator $C = T^M-T^\mu$, as an operator in $\mathcal{R}$,
has an inverse in $[\mathcal{R}]$.
Then the following equality:
$$ \mathbf{R}_\lambda (A) = -\frac{2}{(\lambda+1)^2} R^\mu_{ \frac{1-\lambda}{1+\lambda} } - \frac{1}{\lambda+1} E_H $$
\begin{equation}
\label{f3_32}
+ \frac{2}{(\lambda+1)^2} R^\mu_{ \frac{1-\lambda}{1+\lambda} }
C^{\frac{1}{2}} \mathbf{k}(\lambda) \left(
E_\mathcal{R} + ( \mathbf{Q}_\mu(\lambda)-E_\mathcal{R} ) \mathbf{k}(\lambda)
\right)^{-1}
C^{\frac{1}{2}} R^\mu_{ \frac{1-\lambda}{1+\lambda} },
\end{equation}
where
$\mathbf{Q}_\mu (\lambda) = Q_\mu \left( \frac{1-\lambda}{1+\lambda} \right)$,
$\mathbf{k}(\lambda) = k\left( \frac{1-\lambda}{1+\lambda} \right)$;
$k(\cdot)\in R_{\mathcal{R}}[-1,1]$,
establishes a bijective correspondence between the set $R_{\mathcal{R}}[-1,1]$ and the set
of all generalized $\Pi$-resolvents of $A$.
Here $Q_\mu$ is defined by~(\ref{f3_19}) for $T$, $R^\mu_z = (T^\mu - zE_H)^{-1}$, and
$\mathbf{R}_\lambda (A)$ is a generalized $\Pi$-resolvent of $A$.

Moreover, the canonical resolvents correspond in~(\ref{f3_32}) to the constant functions
$k(z)\equiv K$, $K\in [0,E_{\mathcal{R}}]$.
\end{thm}
\begin{proof}
It follows directly from the preceding considerations, formula~(\ref{f3_27}) and by applying Theorem~\ref{t3_2}.

\end{proof}

Let the strong matrix Stieltjes moment problem be given and conditions~(\ref{f1_3}) hold.
Consider an arbitrary Hilbert space $H$ and a sequence of elements $\{ x_n \}_{n\in \mathbb{Z}}$ in $H$,
such that relation~(\ref{f2_3}) holds. Let $A = \overline{A_0}$, where the operator $A_0$ is
defined by~(\ref{f2_6}). Denote $L_N = \mathop{\rm Lin}\nolimits \{ x_k \}_{k=0}^{N-1}$.
Define a linear transformation $G$ from $\mathbb{C}^N$ onto $L_N$ by the following relation:
\begin{equation}
\label{f3_33}
G \vec u_k = x_k,\qquad k=0,1,...,N-1,
\end{equation}
where $\vec u_k = (\delta_{0,k},\delta_{1,k},...,\delta_{N-1,k})$.
\begin{thm}
\label{t3_4}
Let the strong matrix Stieltjes moment problem~(\ref{f1_1}) be given and conditions~(\ref{f1_3}) be satisfied.
Let $\{ x_n \}_{n\in \mathbb{Z}}$ be a sequence of elements of a Hilbert space $H$ such that
relation~(\ref{f2_3}) holds. Let $A = \overline{A_0}$, where the operator $A_0$ is
defined by relation~(\ref{f2_6}). Let $T = -E_H + 2(E_H+A)^{-1}$. The following statements are true:

\begin{itemize}

\item[1)] If $T^\mu = T^M$, then the moment problem~(\ref{f1_1}) has a unique solution. This solution
is given by
\begin{equation}
\label{f3_34}
M(t) = (m_{j,n}(t))_{j,n=0}^{N-1},\quad
m_{j,n}(t) = ( E^\mu_t x_j,x_n )_H,\
0\leq j,n\leq N-1,
\end{equation}
where $\{ E^\mu_t \}$ is the left-continuous orthogonal resolution of unity of the operator
$A^\mu = -E_H + 2(E_H+T^\mu)^{-1}$.

\item[2)] If $T^\mu \not= T^M$, define the extended operator $T_e$ by~(\ref{f3_17}); $\mathcal{R}_e =
H\ominus D(T_e)$, $C=T^M-T^\mu$, and
$R^\mu_z = (T^\mu - zE_H)^{-1}$,
$Q_{\mu,e} (z) =
\left.\left( C^{\frac{1}{2}} R^\mu_z C^{\frac{1}{2}} + E_H \right)\right|_{\mathcal{R}_e}$,
$z\in \mathbb{C}\backslash[-1,1]$.
An arbitrary solution $M(\cdot)$ of the moment problem can be found by the Stieltjes-Perron inversion formula
from the following relation
$$ \int_{ \mathbb{R}_+ } \frac{1}{t - z} d M^T(t) $$
\begin{equation}
\label{f3_35}
= \mathcal{A}(z) - \mathcal{C}(z) \mathbf{k}(z) (E_{ \mathcal{R}_e } + \mathcal{D}(z) \mathbf{k}(z))^{-1}
\mathcal{B}(z),
\end{equation}
where $\mathbf{k}(\lambda) = k\left( \frac{1-\lambda}{1+\lambda} \right)$,
$k(z)\in R_{\mathcal{R}_e}[-1,1]$, and on the right-hand side one means the matrix of the corresponding
operator in $\mathbb{C}^N$.
Here $\mathcal{A}(z),\mathcal{B}(z),
\mathcal{C}(z),\mathcal{D}(z)$ are analytic operator-valued functions given by
\begin{equation}
\label{f3_36}
\mathcal{A}(z) = -\frac{2}{(\lambda+1)^2} G^* R^\mu_{ \frac{1-\lambda}{1+\lambda} }G - \frac{1}{\lambda+1} G^*G:\
\mathbb{C}^N \rightarrow \mathbb{C}^N,
\end{equation}
\begin{equation}
\label{f3_37}
\mathcal{B}(z) = C^{\frac{1}{2}} R^\mu_{ \frac{1-\lambda}{1+\lambda} } G:\
\mathbb{C}^N \rightarrow \mathcal{R}_e,
\end{equation}
\begin{equation}
\label{f3_38}
\mathcal{C}(z) = \frac{2}{(\lambda+1)^2} G^* R^\mu_{ \frac{1-\lambda}{1+\lambda} } C^{\frac{1}{2}}:\
\mathcal{R}_e \rightarrow \mathbb{C}^N,
\end{equation}
\begin{equation}
\label{f3_39}
\mathcal{D}(z) = Q_{\mu,e}\left( \frac{1-\lambda}{1+\lambda} \right)
- E_{ \mathcal{R}_e }:\ \mathcal{R}_e \rightarrow \mathcal{R}_e.
\end{equation}
Moreover, the correspondence between all solutions of the moment problem and $k(z)\in R_{\mathcal{R}_e}[-1,1]$
is bijective.
\end{itemize}

\end{thm}
\begin{proof}
Consider the case 1). In this case all self-adjoint contractions
$\widetilde T\supseteq T$ in a Hilbert space $\widetilde H\supseteq H$ coincide on $H$ with $T^\mu$,
see~\cite[p. 1039]{cit_400_KO}. Thus, the corresponding sc-spectral functions are
spectral functions of the self-adjoint operator $T^\mu$, as well. However, a self-adjoint operator
has a unique (normalized) spectral function. Thus, a set of sc-spectral functions of $T$ consists of
a unique element. Therefore the set of $\Pi$-resolvents of $A$ consists of a
unique element, as well. This element is the spectral function of $A^\mu$.

\noindent
Consider the case 2).
By Theorem~\ref{t3_2} and relation~(\ref{f3_1}) it follows that an arbitrary solution $M(t) =
(m_{j,n}(t))_{j,n=0}^{N-1}$ of the moment problem~(\ref{f1_1}) can be found from the following relation:
$$ \int_{ \mathbb{R}_+ } \frac{1}{t - z} d m_{j,n}(t) =
(\mathbf{R}_z x_j,x_n)_H,\quad 0\leq j,n\leq N-1;\ z\in \mathbb{C}\backslash \mathbb{R}, $$
where $\mathbf{R}_z$ is a generalized $\Pi$-resolvent of the operator $A$. Moreover, the correspondence between
the set of all generalized $\Pi$-resolvents of $A$ (which is equal to the set of
all generalized $\Pi$-resolvents of $A_e$) and the set of all solutions of the moment problem is bijective.
Notice that $T^\mu = T_e^\mu$ and $T^M = T_e^M$. By Theorem~\ref{t3_3} (applied to the operator $A_e$)
we may rewrite the latter relation in the following form:
$$ \int_{ \mathbb{R}_+ } \frac{1}{t - z} d m_{j,n}(t)
= \left( \left\{
-\frac{2}{(\lambda+1)^2} R^\mu_{ \frac{1-\lambda}{1+\lambda} } - \frac{1}{\lambda+1} E_H \right.\right. $$
$$ \left. \left.
+ \frac{2}{(\lambda+1)^2} R^\mu_{ \frac{1-\lambda}{1+\lambda} }
C^{\frac{1}{2}} \mathbf{k}(\lambda) \left(
E_{ \mathcal{R}_e } + ( \mathbf{Q}_{\mu,e}(\lambda)-E_{ \mathcal{R}_e } ) \mathbf{k}(\lambda)
\right)^{-1}
C^{\frac{1}{2}} R^\mu_{ \frac{1-\lambda}{1+\lambda} }
 \right\} x_j,x_n
 \right)_H, $$
where $\mathbf{k}(\lambda) = k\left( \frac{1-\lambda}{1+\lambda} \right)$, $k(z)\in R_{\mathcal{R}_e}[-1,1]$,
$\mathbf{Q}_{\mu,e}(\lambda) = Q_{\mu,e}\left( \frac{1-\lambda}{1+\lambda} \right)$.
Then
$$ \int_{ \mathbb{R}_+ } \frac{1}{t - z} d m_{j,n}(t)
= \left( \left\{
-\frac{2}{(\lambda+1)^2} G^* R^\mu_{ \frac{1-\lambda}{1+\lambda} }G - \frac{1}{\lambda+1} G^*G
+ \frac{2}{(\lambda+1)^2} G^*
\right.\right. $$
$$ \left. \left.
* R^\mu_{ \frac{1-\lambda}{1+\lambda} } C^{\frac{1}{2}} \mathbf{k}(\lambda) \left(
E_{ \mathcal{R}_e } + ( \mathbf{Q}_{\mu,e}(\lambda)-E_{ \mathcal{R}_e } ) \mathbf{k}(\lambda)
\right)^{-1}
C^{\frac{1}{2}} R^\mu_{ \frac{1-\lambda}{1+\lambda} } G
 \right\} u_j,u_n
 \right)_{\mathbb{C}^N}. $$
Introducing functions
$\mathcal{A}(z),\mathcal{B}(z), \mathcal{C}(z),\mathcal{D}(z)$ by formulas~(\ref{f3_36})-(\ref{f3_39})
one easily obtains relation~(\ref{f3_35}).
\end{proof}

\begin{thm}
\label{t3_5}
Let the strong matrix Stieltjes moment problem~(\ref{f1_1}) be given and conditions~(\ref{f1_3}) be satisfied.
Let $\{ x_n \}_{n\in \mathbb{Z}}$ be a sequence of elements of a Hilbert space $H$ such that
relation~(\ref{f2_3}) holds. Let $A = \overline{A_0}$, where the operator $A_0$ is
defined by relation~(\ref{f2_6}).
The moment problem is determinate if and only if
$T^\mu = T^M$, where $T^\mu$,$T^M$ are the extremal extensions of the operator
$T = -E_H + 2(E_H+A)^{-1}$.
\end{thm}
\begin{proof}
The sufficiency follows from Statement~1 of Theorem~\ref{t3_4}.
The necessity follows from Statement~2 of Theorem~\ref{t3_4}, if we take into account that the class
$R_{\mathcal{R}_e}([-1,1])$, where $\dim \mathcal{R}_e >0$, has at least two different elements.
In fact, from the definition of the class $R_{\mathcal{R}_e}([-1,1])$ it follows that $k_1(z)\equiv 0$, and
$k_1(z)\equiv E_{\mathcal{R}_e}$, belong to $R_{\mathcal{R}_e}([-1,1])$.
\end{proof}

\begin{exm}
Consider the moment problem~(\ref{f1_1}) with $N=2$ and
$$ S_n = \left(
\begin{array}{cc} 1 & \frac{3}{ \sqrt{10} } \\
\frac{3}{ \sqrt{10} } & 1 \end{array}
\right),\qquad n\in \mathbb{Z}. $$
In this case we have
$$ \Gamma = (S_{i+j})_{i,j=-\infty}^\infty = (\gamma_{n,m})_{n,m=-\infty}^\infty, $$
where
$$ \gamma_{2k,2l} = \gamma_{2k+1,2l+1} = 1,\ \gamma_{2k,2l+1} = \gamma_{2k+1,2l} = \frac{3}{ \sqrt{10} },\qquad
k,l\in \mathbb{Z}. $$
Consider the space $\mathbb{C}^2$ and elements $u_0,u_1\in \mathbb{C}^2$:
$$ u_0 = \frac{1}{ \sqrt{2} } (1,1),\quad  u_1 = \frac{1}{ \sqrt{5} } (1,2). $$
Set
$$ x_{2k} = u_0,\quad x_{2k+1} = u_1,\qquad k\in \mathbb{Z}. $$
Then relation~(\ref{f2_3}) holds.
Define by~(\ref{f2_6}) the operator $A_0$. In this case $A=A_0 = E_{\mathbb{C}^2}$.
Therefore the operators $A$ and $T=-E_H + 2(E_H+A)^{-1}$ are self-adjoint and have  unique spectral functions.
Hence, $T^M=T^\mu$, and by Theorem~\ref{t3_5} we conclude that
the moment problem has a unique solution.
By Theorem~\ref{t3_1} it has the following form
$$ M(\lambda) = (m_{k,j} (\lambda))_{k,j=0}^{N-1},\quad
m_{k,j} (\lambda) = ( \mathbf E_\lambda x_k, x_j)_H,
$$
where $\mathbf E_\lambda$ is the left-continuous spectral function of the operator $E_{\mathbb{C}^2}$.
Consequently, the matrix function $M(t)$ is equal to $0$, for $t\leq 1$,
and $M(t) = \left(
\begin{array}{cc} 1 & \frac{3}{ \sqrt{10} } \\
\frac{3}{ \sqrt{10} } & 1 \end{array}
\right)$, for $t>1$.

\end{exm}

\begin{center}
{\large\bf The strong matrix Stieltjes moment problem.}
\end{center}
\begin{center}
{\bf A.E. Choque Rivero, S.M. Zagorodnyuk}
\end{center}

In this paper we study the strong matrix Stieltjes moment problem.
We obtain necessary and sufficient conditions for its
solvability. An analytic description of all
solutions of the moment problem is derived.
Necessary and sufficient conditions for the determinateness of the moment problem
are given.

}
\end{document}